\documentclass{article}
\usepackage{amsmath, amssymb, amsthm}
\title{\textbf{Generalizing Parking Functions with Randomness}}
\author{Melanie Tian and Enrique Trevi\~no}

\newtheorem{theorem}{Theorem}

\newtheorem{prop}{Proposition}

\newtheorem{lemma}{Lemma}

\begin{document}
\maketitle

\begin{abstract}
 Consider $n$ cars $C_1, C_2, \ldots, C_n$ that want to park in a parking lot with parking spaces $1,2,\ldots,n$ that appear in order. Each car $C_i$ has a parking preference $\alpha_i \in \{1,2,\ldots,n\}$. The cars appear in order, if their preferred parking spot is not taken, they take it, if the parking spot is taken, they move forward until they find an empty spot. If they do not find an empty spot, they do not park. An $n$-tuple $(\alpha_1, \alpha_2, \ldots, \alpha_n)$ is said to be a parking function, if this list of preferences allows every car to park under this algorithm. For an integer $k$, we say that an $n$-tuple is a $k$-Naples parking function if the cars can park with the modified algorithm, where park $C_i$ backs up $k$-spaces (one by one) if their spot is taken before trying to find a parking spot in front of them. We introduce randomness to this problem in two ways: 1) For the original parking function definition, for each car $C_i$ that has their preference taken, we decide with probability $p$ whether $C_i$ moves forwards or backwards when their preferred spot is taken; 2) For the $k$-Naples definition, for each car $C_i$ that has their preference taken, we decide with probability $p$ whether $C_i$ backs up $k$ spaces or not before moving forward. In each of these models, for an $n$-tuple $\alpha\in\{1,2,\ldots,n\}^n$, there is now a probability that the model ends in all cars parking or not. For each of these random models, we find a formula for the expected value. Furthermore, for the second random model, in the case $k =1$, $p=1/2$, we prove that for any $1\le t\le 2^{n-2}$, there is exactly one $\alpha\in\{1,2,\ldots,n\}^n$ such that the probability that $\alpha$ parks is $(2t-1)/2^{n-1}$.
\end{abstract}

\section{Introduction}
Consider $n$ cars $C_1, C_2, \ldots, C_n$ that want to park in a parking lot with parking spaces $1,2,\ldots,n$ that appear in order. Each car $C_i$ has a parking preference $\alpha_i \in \{1,2,\ldots,n\}$. The cars appear in order, if their preferred parking spot is not taken, they take it, if the parking spot is taken, they move forward until they find an empty spot. If they do not find an empty spot, they do not park. An $n$-tuple $(\alpha_1, \alpha_2, \ldots, \alpha_n)$ is said to be a parking function, if this list of preferences allows every car to park under this algorithm. That the number of parking functions is $(n+1)^{n-1}$ was first proved by Konheim and Weiss \cite{Konheim} using generating functions and then Pollak \cite{PollakProof}\footnote{Pollak's proof was published by Riordan instead of Pollak.} using a clever bijection. For more details on parking functions, the survey paper by Yan \cite{Yan} is a great resource.

One generalization of parking functions, introduced by Baumgardner \cite{Baumgardner}, concerns what's known as Naples parking functions. In this scenario, the cars that reach a taken parking spot, back up one spot before moving forward. For example, the tuple $(4,3,3,1)$ is not a parking function, but it is a Naples parking function. One could then generalize the concept to $k$-Naples parking functions, where cars back up $k$ spots before moving forward.  Christensen et. al. \cite{Harris-Naples} found the following recursive formula to find the number $N_k(n+1)$ of $k$-Naples parking functions on $n+1$ cars:
\begin{equation}\label{naples-orig}
N_{k}(n) = \sum_{i=0}^{n-1} {n-1\choose i}N_k(i)(n-i)^{n-i-2}(i+1+\min\{k,n-i-1\}).
\end{equation}

In a playful paper by Carlson, Christensen, Harris, Jones and Ramos Rodr\'iguez \cite{Harris-adventure}, the authors introduce many variants of parking functions and mention many open problems. In particular, chapter 1.9 of the paper suggests introducing randomness to Naples parking functions. Randomness had been considered by Diaconis and Hicks \cite{Diaconis} where they study whether a random parking function has certain properties. In our case, inspired by Carlson et. al, we consider models where we introduce randomness to the parking model (as opposed to randomly choosing a tuple and checking properties of it). The two random versions of parking functions we consider are
\begin{enumerate}
\item Suppose the spot $C_i$ wants is taken, then $C_i$ continues forward with probability $p$ or changes direction and goes backwards with probability $1-p$.
\item Suppose the spot $C_i$ wants is taken, then $C_i$ backs up $k$ spaces with probability $p$ before moving forward or just continues forward with probability $1-p$.
\end{enumerate}

For the first of these two models, the original paper on parking functions by Konheim and Weiss suggests that the expected number of parking functions is also $(n+1)^{n-1}$. We will explore this model and give a different proof of this result in section \ref{model 1}. The second model is the one we study more. In section \ref{random naples} we will prove a generalization of \eqref{naples-orig}, namely
\begin{theorem}\label{recursion thm}
Let $T_{k,p}(n)$ be the expected number of random $k$-Naples parking functions, then for $n\ge 1$,
\begin{equation}\label{recursive eq}
T_{k,p}(n)=\sum_{i=0}^{n-1}{{n-1\choose i}T_{k,p}(i){(n-i)}^{n-i-2}}(i+1+p\min\{k,n-i-1\}).
\end{equation}
\end{theorem}
For example, Table \ref{evs}, includes the expected number $T(n)$ of $1$-Naples parking functions when $p = 1/2$,
\begin{table}[!ht]
\begin{center}
\begin{tabular}{|c|c|c|c|c|c|c|c|c|c|}\hline
$n$&1&2&3&4&5&6&7&8\\ \hline
$P(n)$         &1&3&16&125&1296&16807&262144& 4782969\\ \hline
$T(n)$         &1&3.5&20&163.25&1744.25&23121.375&366699& 6779029.0625\\ \hline
$\frac{N(n)+P(n)}{2}$&1&3.5&20&164&1760.5&23437&373107.5& 6920142.5\\ \hline
$N(n)$         &1&4&24&203&2225&30067&484071& 9057316\\ \hline
\end{tabular}
\end{center}
\caption{$N(n)$ is the number of Naples-parking functions, $P(n)$ is the number of parking functions and $T(n)$ is the expected number of random Naples parking functions (when $p=1/2$).}\label{evs}
\end{table}
We will investigate the case $k =1$ and $p=1/2$ deeply in section \ref{cool section}, where we study the distribution of the probabilities for different $n$-tuples and in particular we prove the following surprising result:
\begin{theorem}\label{favorite thm}
Given $n$ cars, there is one and only one parking preference for which the probability that every car parks is $\frac{2t-1}{2^{n-1}}$, where $t\in [1,2^{n-2}]$.
\end{theorem}

Finally, in section \ref{section ineqs} we study some asymptotics comparing the expected number of random $1$-Naples parking function to 1-Naples parking functions and to the usual parking function model. Namely, we will prove the following theorem
\begin{theorem}\label{thm comparing} Let $N(n) = N_1(n)$ be the number of Naples parking functions for $n$ cars, $P(n) = (n+1)^{n-1}$ be the number of parking functions, and $T(n) = T_{1/2,1}(n)$ be the expected number of random Naples parking functions for $n$ cars, with probability $1/2$ of going back one spot. Then
$$P(n) \le T(n)\le \frac{N(n)+P(n)}{2}.$$
\end{theorem}

\section{Random Direction Parking}\label{model 1}
Here we will consider the model where if spot $\alpha_i$ is taken for car $C_i$, then with probability $p$ they look for a parking spot ahead of them and with probability $1-p$ they look for a parking spot behind them. For example, the tuple $(1,2,2,1)$ has probability $p^2$ of parking. Let $R_p(n)$ be the expected number of parking functions, i.e.,
$$R_p(n) = \sum_{\alpha\in \{1,2,\ldots,n\}^n} \mathbb{P}(\alpha \text{ parks}).$$ The following theorem was first proved by Konheim and Weiss, but we include a different proof here.

\begin{theorem} For any positive integer $n$, $R_p(n) = (n+1)^{n-1}$.
\end{theorem}

\begin{proof} Consider an $n$-tuple $\alpha\in\{1,2,\ldots, n+1\}^n$. Consider having $n+1$ parking spots in a circle, so the spot $n+1$ represents that someone didn't park. Let $v(\alpha)$ be the vector where the $i$-th entry is the probability that parking spot $i$ is not taken under this circular random parking model. Note that if you shift every entry in $\alpha$ by 1 modulo $n+1$, then the distribution of $v(\alpha)$ shifts by 1 to the right. Therefore (since the sum of entries in each $v(\alpha)$ is 1),
$$\sum_{\alpha\in\{1,2,\ldots,n+1\}^n} v(\alpha) = \frac{1}{n+1}\left((n+1)^n, (n+1)^n, \ldots, (n+1)^n\right).$$
Therefore, if you consider $n$-tuples $\alpha\in\{1,2,\ldots,n+1\}^n$, the expected number would be $(n+1)^{n-1}$ by linearity of expectation. We want to prove the expectation is $(n+1)^{n-1}$ when only considering tuples $\alpha\in\{1,2,\ldots,n\}^n$. The key observation is that if $\alpha\in \{1,2,\ldots,n+1\}$ and $\alpha\not\in\{1,2,\ldots,n\}$, then at least one car wants to park in $n+1$ and whoever gets there first takes the parking spot. Therefore, $n+1$ is not empty for those $\alpha$'s and they don't contribute anything to $R_p(n)$.

Therefore $R_p(n) = (n+1)^{n-1}$.
\end{proof}

The next result is a statement about the probability that a particular $n$-tuple ``parks''.
\begin{prop}
Let $\alpha\in\{1,2,\ldots,n\}^n$. If $\alpha$ is a permutation of $\{1,2,\ldots,n\}$, then $\mathbb{P}(\alpha \text{ parks}) = 1$, otherwise
$$0 < \mathbb{P}(\alpha \text{ parks}) < 1.$$
\end{prop}
\begin{proof}
When $\alpha$ is a permutation of $\{1,2,\ldots, n\}$, then every car takes their preferred spot and everyone parks, so $\mathbb{P}(\alpha \text{ parks}) = 1$. We can now assume $\alpha$ is not a permutation of $\{1,2,\ldots,n\}$. Therefore, there is at least one pair of cars that want the same parking spot. We will first show $\mathbb{P}(\alpha \text{ parks}) > 0$. Note that anytime a car reaches a preferred spot that is taken, then at least one of the two directions has an empty spot, so the car can choose that direction and park. This is true for any car, therefore $\mathbb{P}(\alpha \text{ parks}) > 0$.

We will now prove that $\mathbb{P}(\alpha \text{ parks}) < 1$. Suppose a tuple $\alpha$ parked after a sequence of correct choices. Now consider all possible choices that lead to a successful parking configuration for that tuple $\alpha$. Among these finitely many choices, consider one that had a car $C_i$ with $i$ as large as possible, where $C_i$'s preferred parking spot is taken. Take the same choices that led to this parking spot, except that when one reaches $C_i$, the car $C_i$ goes in the opposite direction. If $C_i$ finds a parking spot, then there is a $j > i$ where $C_j$ had taken that spot under the previous choice, so there is a larger $j$ with $C_j$ reaching a conflict. The only alternative is that $C_i$ didn't park. Therefore, $\mathbb{P}(\alpha \text{ parks}) < 1$.
\end{proof}

\section{Generalizing $k$-Naples with an Unfair Coin}\label{random naples}
In this section we are considering the parking model where there are $n$ cars, $C_1, C_2, \ldots, C_n$, where car $C_i$ prefers to park at spot $\alpha_i$. Each car goes in turns, $C_1$ first, $C_2$ second, etc., where car $C_i$ parks at $\alpha_i$ if the spot is not taken. If $\alpha_i$ is taken, then with probability $p$, $C_i$ moves $k$ spaces backwards (if the number of spaces behind is less than $k$, then that's how far behind they go) before looking for a parking spot in front of them. To illustrate the model, Table \ref{table: random-Naples example} shows the probability that a $3$-tuple of preferences parks when $k = 1$.

\begin{table}[!ht]
\begin{center}
\begin{tabular}{|c|c|c|c|c|c|c|c|c|c|}
  \hline
  $\alpha$                           &  $111$  & $112$  & $113$ & $121$ & $122$ & $123$ & $131$ & $132$ & $133$ \\ \hline
  $\mathbb{P}(\alpha \text{ parks})$ &    1    &  1     &   1   &   1   &  1    &   1   &  1    &    1  &  $p$  \\ \hline
  $\alpha$                           &  $211$ & $212$  & $213$ & $221$ & $222$         & $223$  & $231$ & $232$ & $233$  \\ \hline
  $\mathbb{P}(\alpha \text{ parks})$ &    1   &   1    &   1   &   1   &  $1-(1-p)^2$  &  $p$   &   1   &  $p$     & 0 \\ \hline
  $\alpha$                           &  $311$ & $312$  & $313$ & $321$ & $322$ & $323$ & $331$  & $332$ & $333$  \\ \hline
  $\mathbb{P}(\alpha \text{ parks})$ &    1   &   1    &  $p$  &    1  &  $p$  &   0   &  $p$   &  $p^2$&  0\\   \hline
\end{tabular}
\end{center}
\caption{All 3-tuples $\alpha$ of parking preferences (written $\alpha_1\alpha_2\alpha_3$ instead of $(\alpha_1,\alpha_2,\alpha_3)$) and the probability that they park under the random 1-Naples parking model.}\label{table: random-Naples example}
\end{table}

Let $T_{k,p}(n)$ be the expected number of parking functions under the random $k$-Naples model, i.e.,
$$T_{k,p}(n) = \sum_{\alpha\in\{1,2,\ldots,n\}^n} \mathbb{P}(\alpha \text{ parks}).$$

\begin{proof}[Proof of Theorem \ref{recursion thm}]
The set of all successful parking processes for $n$ cars can be partitioned into $n$ sets depending on where the $n$-th car eventually parks. Suppose the $n$-th car eventually parks at spot $i+1$, where $i\in[0,n-1]$. Right before the $n$-th car is going to park, spot $i+1$ has to be open, and $i$ cars have taken the $i$ spots left of spot $i+1$, and $n-i-1$ cars have taken the $n-i-1$ spots to the right of spot $i+1$.

First we choose $i$ cars that take the left $i$ spots, there are ${n-1\choose i}$ ways of doing so. Then the expected number of preferences for them to park accordingly is $T_{k,p}(i)$. The expected number of preferences for $n-i-1$ cars to take the right $n-i-1$ spots is $(n-i)^{n-i-2}$ because this is the usual parking function situation. Finally the $n$-th car could have any preference from 1 to $i+1$ which guarantees they park at location $i+1$ or they could have preference between $i+2$ and $\min{\{i+k+1,n\}}$ and flipped a coin to back up $k$ spots, which has probability $p$ of happening. This contributes a factor of $(i+1+p\min\{k,n-i-1\})$. Since $i$ ranges from 0 to $n-1$, using linearity of expectation we get our desired sum.
\end{proof}

\section{Distribution of Random 1-Naples}\label{cool section}
In this section we dive deeper into the random Naples parking in the case when $k =1$ and $p=1/2$. In particular we study the distribution of the probabilities for different $n$-tuples. For a real number $q\in[0,1]$, let $f(q)$ be the number of $\alpha\in\{1,2,\ldots,n\}$ such that $\mathbb{P}(\alpha \text{ parks}) = q$. For example, Table \ref{table distribution}, finds all relevant $f(q)$'s for $n=7$ other than $f(1)$.
\begin{table}[!ht]
\begin{center}
\begin{tabular}{|c|c|c|c|c|c|c|c|c|}\hline
$q$& 0 & 1/64 & 2/64 & 3/64 & 4/64 & 5/64 & 6/64 & 7/64  \\ \hline
$f(q)$& 339472& 1& 136& 1& 2194& 1& 209& 1\\ \hline
$q$&8/64& 9/64 & 10/64 & 11/64 & 12/64 & 13/64 & 14/64 & 15/64  \\ \hline
$f(q)$&12466&1& 140& 1& 3107& 1& 143& 1\\ \hline
$q$&16/64& 17/64 & 18/64 & 19/64 & 20/64 & 21/64 & 22/64 & 23/64  \\ \hline
$f(q)$&40610& 1& 141& 1& 1361& 1& 74& 1\\ \hline
$q$&24/64& 25/64 & 26/64 & 27/64 & 28/64 & 29/64 & 30/64 & 31/64  \\ \hline
$f(q)$&14253& 1& 75&1& 1589& 1& 148& 1\\ \hline
$q$&32/64&33/64 & 34/64 & 35/64 & 36/64 & 37/64 & 38/64 & 39/64  \\ \hline
$f(q)$&94792& 1& 30& 1& 1171& 1& 33&1\\ \hline
$q$&40/64& 41/64 & 42/64 & 43/64 & 44/64 & 45/64 & 46/64 & 47/64  \\ \hline
$f(q)$&4861&1& 104& 1& 576& 1& 37& 1\\ \hline
$p$&48/64& 49/64 & 50/64 & 51/64 & 52/64 & 53/64 & 54/64 & 55/64  \\ \hline
$f(q)$&35324&1& 35& 1& 614& 1& 38& 1\\ \hline
$p$&56/64& 57/64 & 58/64 & 59/64 & 60/64 & 61/64 & 62/64 & 63/64  \\ \hline
$f(q)$&6819&1& 39& 1& 734& 1& 42& 1\\ \hline
\end{tabular}
\end{center}
\caption{Distribution of probability for $n=7$, $q$ for probability and $f(q)$ for number of preferences of probability $q$.  }\label{table distribution}
\end{table}

There are several patterns that appear in Table \ref{table distribution} that we are able to prove are satisfied, namely, we will prove
\begin{enumerate}
    \item\label{pattern 1} We have that $f(1)$ is the number of parking functions of length $n$, i.e., $f(1) = (n+1)^{n-1}$.
    \item\label{pattern 2} We have that $f(0)$ is the number of tuples that are not 1-Naples parking functions.
    \item\label{pattern 3} We have $f(q) > 0$ if and only if $q$ can be written as $a/b$ with $b=2^{n-1}$ and $0\le a\le 2^{n-1}$.
    \item\label{pattern 4} We have $f(q) = 1$ if and only if $q$ can be written as $a/b$ with $b = 2^{n-1}$ and $1\le a\le 2^{n-1}$ odd.
\end{enumerate}
The first two items show that this random model in some sense measures how close to a parking function a tuple was originally, and it shows that only tuples that ``park'' under the Naples model get positive probability. Our favorite (and hardest to prove) of these patterns is the fourth one, which we called Theorem \ref{favorite thm} in the Introduction.

Before we can prove these patterns, let's translate the probabilistic model into a counting problem. For an $n$-tuple $\alpha$, we want to calculate $\mathbb{P}(\alpha)$. Consider all tuples $(b_2,b_3,\ldots, b_n)\in\{0,1\}^{n-1}$. If car $C_i$'s preferred spot $\alpha_i$ is taken, we can simulate the model by saying that $C_i$ moves forward if $b_i = 1$ and it goes one spot back before moving forward if $b_i = 0$. If we let $g(\alpha)$ be the number of tuples $(b_2,\ldots,b_n)$ for which $\alpha$ parks, then
\begin{equation}\label{eq: counting}
\mathbb{P}(\alpha \text{ parks}) = \frac{g(\alpha)}{2^{n-1}}.
\end{equation}

The following lemma is crucial in proving patterns \ref{pattern 1} and \ref{pattern 2}
\begin{lemma}\label{main lemma}
Let $n$ be a positive integer. Suppose that $\alpha = (\alpha_1,\alpha_2,\ldots,\alpha_n)\in\{1,2,\ldots,n\}^n$ parks for the choices $\beta=(\beta_2,\beta_3,\ldots,\beta_n)\in\{0,1\}^{n-1}$. If any $\beta_i = 1$, then $\alpha$ also parks when $\beta_i=1$ is replaced by $\beta_i' = 0$.
\end{lemma}
\begin{proof} Let $A_1, A_2, \ldots, A_n$ be the parking locations of cars $C_1, C_2, \ldots, C_n$ when the parking preference is $\alpha$ and the choices on whether continuing forwards or taking back a step are $\beta$. Now replace $\beta_i$ from 1 to 0. Suppose $\alpha$ doesn't park under these new conditions. Let $e$ be the last spot where no one parked, and let $C_j$ be the last car that doesn't park. Since $C_j$ doesn't park, it means there is a car $C_{k_1}$ with $k_1<j$ that parked at $A_j$. Then $e<\alpha_{k_1}\le A_j-1$ or $\alpha_{k_1} = A_j+1$ because if $\alpha_{k_1}\le e$, then $C_{k_1}$ would not leave $e$ empty, if $C_{k_1}= A_j$ then $A_j$ wouldn't be open for $C_j$ under the initial configuration, and if $C_{k_1} \ge A_j+2$, then by definition $C_{k_1}$ cannot take a place before $A_j+1$. We will prove $\alpha_{k_1} \ne A_j+1$. If $\alpha_{k_1} = A_j+1$, then $A_j+1$ must be taken by another car, say $C_{k_2}$ with $k_2< k_1$, but since $A_{k_1}\ne A_j$ and $k_1 < j$, $A_{k_1} = A_j+1$. This implies that $\alpha_{k_2} = A_j+2$. The same reasoning implies there exists a car $C_{k_3}$ with $k_3<k_2$ such that $\alpha_{k_3} = A_j+3$, and so on. But this leads to a contradiction because, eventually, you run out of cars or you run out of parking spots. Therefore $e< \alpha_{k_1} \le A_j-1$. Since $A_{k_1} \ne A_j$, that means that there must have been a car $C_{k_2}$ that took spot $A_{k_1} \le A_j-1$. The same reasoning as for $\alpha_{k_1}$ shows that $e<\alpha_{k_2} \le A_{k_1} -1 \le A_j-2$. But then there exist cars $C_{k_3}, C_{k_4}, \ldots$ such that $e<\alpha_{k_3} \le A_j-3,$ $e<\alpha_{k_4} \le A_j-4$, etcetera. This leads to a contradiction as there are only finitely many cars.
\end{proof}

With the lemma in hand, we can prove the following two theorems:

\begin{theorem}\label{thm 1-> parking}
For a positive integer $n$, $\alpha\in\{1,2,\ldots, n\}^n$ is a parking function if and only if
$$\mathbb{P}(\alpha \text{ parks}) = 1.$$
\end{theorem}
\begin{proof} Suppose $\mathbb{P}(\alpha \text{ parks}) = 1$, then, in particular, $\alpha$ parks when $b_2=b_3=\ldots=b_n = 1$, which implies $\alpha$ is a parking function.

Now, let's assume $\alpha$ is a parking function. Therefore $\beta = (1,1,\ldots,1)\in\{0,1\}^{n-1}$ leads $\alpha$ to park. From Lemma \ref{main lemma}, that means that we can replace any 1 by a 0. Repeating this process for all replacements, we can get to any $\beta'\in\{0,1\}^{n-1}$. Therefore, $\mathbb{P}(\alpha \text{ parks}) = 1$.
\end{proof}

\begin{theorem}\label{th, 0<-> not naples}
For a positive integer $n$, $\alpha\in\{1,2,\ldots,n\}^n$ is not a Naples parking function if and only if
$$\mathbb{P}(\alpha \text{ parks})=0.$$
\end{theorem}
\begin{proof} Suppose that $\alpha$ is a Naples-parking function. Then $\beta = (0,0,\ldots,0)\in\{0,1\}^{n-1}$ is a set of choices that lead to parking. Therefore $\mathbb{P}(\alpha \text{ parks}) > 0$.

Now, suppose $\mathbb{P}(\alpha \text{ parks}) > 0$. Then there is a $\beta\in\{0,1\}^{n-1}$ such that $\alpha$ parks when making the choices from $\beta$. By Lemma \ref{main lemma}, we can change all the 1's to 0's, which implies that $\beta' = (0,0,\ldots,0)\in\{0,1\}^{n-1}$ also leads to $\alpha$ parking. Therefore $\alpha$ is a Naples parking function.
\end{proof}

The third pattern will be proved after we prove Theorem \ref{favorite thm}, but one direction is an easy consequence of \eqref{eq: counting}, because it follows that $f(q) > 0$ implies $q$ can be written as $q = a/b$ with $b = 2^{n-1}$ and $0\le a\le 2^{n-1}$.

We will now embark on the proof of our favorite result, pattern \ref{pattern 4}. To give a feel of how the proof will go, Table \ref{table: motivating favorite theorem} shows the $6$-tuples $\alpha$ that have probability $a/b$ with $a$ odd and $b = 2^{5}$.

\begin{table*}[h!]
\begin{center}
\begin{tabular}{|c|c|c|c|c|}\hline
$\alpha$ & (6,6,5,4,3,2) & (5,5,5,4,3,2) & (5,5,4,4,3,2) & (4,4,4,4,3,2) \\ \hline
$\mathbb{P}(\alpha \text{ parks})$ & $1/32$ & $3/32$ & $5/32$ & $7/32$  \\ \hline
$\alpha$ &(5,5,4,3,3,2) & (4,4,4,3,3,2) & (4,4,3,3,3,2) & (3,3,3,3,3,2) \\ \hline
$\mathbb{P}(\alpha \text{ parks})$ & $9/32$ & $11/32$ & $13/32$ & $15/32$ \\ \hline
$\alpha$ & (5,5,4,3,2,2) & (4,4,4,3,2,2) & (4,4,3,3,2,2)& (3,3,3,3,2,2) \\ \hline
$\mathbb{P}(\alpha \text{ parks})$ & $17/32$ & $19/32$ & $21/32$ & $23/32$  \\ \hline
$\alpha$ &(4,4,3,2,2,2) & (3,3,3,2,2,2) & (3,3,2,2,2,2) & (2,2,2,2,2,2) \\ \hline
$\mathbb{P}(\alpha \text{ parks})$ & $25/32$ & $27/32$ & $29/32$ & $31/32$   \\ \hline
\end{tabular}
\end{center}
\caption{16 parking preferences $\alpha$ with $\mathbb{P}(\alpha \text{ parks}) = \frac{k}{32}$ with $k$ is odd.  }\label{table: motivating favorite theorem}
\end{table*}

Note how the tuples have a very particular shape, namely \\ $\alpha_1=\alpha_2\ge\alpha_3\ge\cdots\ge\alpha_6 = 2$, where $\alpha_{i+1}\in\{\alpha_i-1,\alpha_i\}$.

The following lemma studies $g(\alpha)$ when $\alpha$ is of the form described above. This result will be useful in the proof of Theorem \ref{favorite thm}.

\begin{lemma}\label{lemma for favorite theorem}
 Suppose that cars $C_1,C_2,\ldots, C_n$ have preferred parking spots
 $$\alpha=(\underbrace{t,t,\ldots,t}_{m_t},\underbrace{t-1,t-1,\ldots,t-1}_{m_{t-1}},\ldots,\underbrace{3,3,\ldots,3}_{m_3},\underbrace{2,2,\ldots,2}_{m_2}),$$ where there are $m_2$ 2's, $m_3$ 3's, $\ldots$, $m_t$ $t$'s. Then
\begin{align*}
g(\alpha) = \,&\,2^{n-1} - 2^{n-m_2}+2^{n-m_2-1}-2^{n-m_2-m_3}\\&+2^{n-m_2-m_3-1}-2^{n-m_2-m_3-m_4}+\ldots+2^{m_t-1}-1.
\end{align*}
\end{lemma}

\begin{proof} If they are all 2's, then $g(\alpha) = 2^{n-1}-1$, because the only way $\alpha$ doesn't park is if $\beta=(1,1,\ldots,1)$. This covers the case $n=2$. We'll prove it by induction. Suppose the formula is accurate for $m$ cars for any $m\le n-1$. Now consider $\alpha =  (\underbrace{t,t,\ldots,t}_{m_t},\underbrace{t-1,t-1,\ldots,t-1}_{m_{t-1}},\ldots,\underbrace{3,3,\ldots,3}_{m_3},\underbrace{2,2,\ldots,2}_{m_2})$. If $t=2$, they are all 2's and we are done. Suppose $t\ge 3$. Now, note that after the first $n-m_2$ cars park, they either end up in spots $\{2,3,4,\ldots,n-m_2+1\}$ or in spots $\{3,4,\ldots,n-m_2+2\}$. That's because the numbers are consecutive, so there's never a bigger gap. Note that to end up in $\{2,3,4,\ldots,n-m_2+1\}$ is the equivalent of transforming $(\underbrace{t,t,\ldots,t}_{m_t},\underbrace{t-1,t-1,\ldots,t-1}_{m_{t-1}},\ldots,\underbrace{3,3,\ldots,3}_{m_3})$ into $\alpha'=(\underbrace{t-1,t-1,\ldots,t-1}_{m_t},\underbrace{t-2,t-2,\ldots,t-2}_{m_{t-2}},\ldots,\underbrace{2,2,\ldots,2}_{m_3})$ and seeing if $\alpha'$ parks. Now note that if you reach the last $m_2$ 2's with spots $\{2,3,\ldots, n-m_2+1\}$ taken, there are $2^{m_2}-1$ ways of parking (if any of the 2's backs up, $\alpha$ parks, otherwise it doesn't), whereas if you reach the last $m_2$ 2's with spots $\{3,4,\ldots,n-m_2+2\}$ taken, then there are $2(2^{m_2-1}-1)$ ways to park because the first car with preference 2 takes spot 2 automatically, and so whether the car wanted to go back or forward is irrelevant, contributing a factor of 2, and the the rest now succeed if any of the $m_2-1$ cars with preference 2 backs up, which contributes a factor of $2^{m_2-1}-1$. Therefore
$$g(\alpha) = g(\alpha')(2^{m_2}-1) + (2^{n-m_2-1}-g(\alpha'))(2^{m_2}-2) = 2^{n-1}-2^{n-m_2}+g(\alpha').$$
The claim follows from applying the induction hypothesis on $g(\alpha')$.
\end{proof}

We are now ready to prove Theorem \ref{favorite thm}.

\begin{proof}[Proof of Theorem \ref{favorite thm}] Table \ref{table: motivating favorite theorem} shows the Theorem is true for $n=6$ and one can easily check that it's also true for $n <  6$. We may assume $n>6$. Define $g=g(\alpha)$ as in \eqref{eq: counting}. We will first prove that if $g$ is odd, then $\alpha_1 = \alpha_2, \alpha_{i+1}=\alpha_i$ or $\alpha_{i+1} = \alpha_i-1$, and $\alpha_n=2$. Suppose $\alpha_2\ne \alpha_1$, then $C_1$ and $C_2$ parks in $\alpha_1$ and $\alpha_2$, respectively. Therefore, the value of $\beta_2$ is irrelevant, which implies that $g$ is even. Therefore $\alpha_1=\alpha_2$. We also know that $\alpha_i\ge 2$ because if for any $i\ge 2$ we have $\alpha_i=1$, then $\beta_i=0$ and $\beta_i=1$ both mean that the car searches forward, which implies $g$ is even. Consider $\alpha_3$. We know $C_1$ parks at $\alpha_1$, for $\beta_2=0$, $C_2$ parks at $\alpha_1+1$, and for $\beta_2=1$, $C_2$ parks at $\alpha_2-1$. Note that if $|\alpha_3-\alpha_2|\ge2$, then $\beta_3$ is irrelevant, which would then make $g$ even. If $\alpha_3 = \alpha_2+1$, then in the case where $C_2$ parks at $\alpha_2-1$ we have that $C_3$ parks at $\alpha_2+1$ without $\beta_3$ mattering, and in the case that $C_2$ parks at $\alpha_2+1$, $C_3$ parks at $\alpha_2+2$ (or doesn't park if $\alpha_2>n-2$), which means $\beta_3$ is irrelevant. Therefore $g$ is even whenever $\alpha_3\ge \alpha_2+1$ and when $\alpha_3\le \alpha_2-2$. Therefore $\alpha_3\in\{\alpha_2-1,\alpha_2\}$. Now suppose that for all $1\le i\le m-1$ we have $\alpha_{i+1}\in\{\alpha_i-1,\alpha_i\}$. Let's prove that $\alpha_{m+1}\in\{\alpha_m-1,\alpha_m\}$. The first $m$ cars are parks at $\{\alpha_m-1,\alpha_m,\ldots,\alpha_m+m-2\}$ or at $\{\alpha_m,\alpha_m+1,\ldots,\alpha_m+m-1\}$. If $\alpha_{m+1}\in\{\alpha_m+1,\ldots,\alpha_m+m-2\}$, then $C_{m+1}$'s preferred spot is taken and the spot behind them is also taken, so $\beta_{m+1}$ is irrelevant, which implies $g$ is even. If $\alpha_{m+1} = \alpha_m+m-1$, then in the first case $C_{m+1}$ parks at $\alpha_m+m-1$ and $\beta_{m+1}$ is irrelevant, or in the second case, since $\alpha_m+m-1$ and $\alpha_m+m-2$ are both taken, $C_{m+1}$ parks at $\alpha_m+m$ (or fails to park), in which case $\beta_{m+1}$ is irrelevant, forcing $g$ to be even. Therefore $\alpha_{m+1}\ne \alpha_m+m-1$. If $\alpha_{m+1} > \alpha_m+m-1$, then $C_{m+1}$ parks at $\alpha_{m+1}$ and $\beta_{m+1}$ is irrelevant. Therefore $\alpha_{m+1}\le \alpha_m$. If $\alpha_{m+1}\le \alpha_m-2$, then $C_{m+1}$ parks at $\alpha_{m+1}$ and $\beta_{m+1}$ is irrelevant, which means $g$ is even. Therefore $\alpha_{m+1}\in\{\alpha_m-1,\alpha_m\}$. Finally, to complete the proof of our claim, we need to prove $\alpha_n = 2$. Suppose $\alpha_n\ge 3$, then nobody parked at position 1 since $\alpha_i\ge \alpha_n$ for all $i<n$.

Therefore, if $g$ is odd, then
\begin{equation}\label{eq: inequality}
\alpha_1=\alpha_2 \ge \alpha_3\ge\cdots\ge\alpha_n = 2 \ \ \ \text{with } \alpha_{i+1}\in\{\alpha_i-1,\alpha_i\}.
\end{equation}
There are $2^{n-2}$ $\alpha$'s satisfying \eqref{eq: inequality} because $\alpha_n = 2$ and $\alpha_{n-1}$ has two choices $\{2,3\}$, then $\alpha_{n-2}$ has two choices $\{\alpha_{n-1},\alpha_{n-1}+1\}$, and so on, until $\alpha_2$ has two choices $\{\alpha_3,\alpha_3+1\}$ while $\alpha_1 = \alpha_2$. Note that all elements had those choices because if you increase one by one you get $\alpha_1=\alpha_2 = n, \alpha_3 = n-1, \ldots, \alpha_{n-1}=3,\alpha_n = 2$. So if $g(\alpha)$ is odd, then $\alpha$ is one of these $2^{n-2}$ possible $n$-tuples.

Let $A$ be the set of $\alpha\in\{1,2,\ldots,n\}^n$ satisfying \eqref{eq: inequality}. We know $|A| = 2^{n-2}$. Let $\alpha\in A$. Then $\alpha$ is of the form
\begin{equation}\label{alpha1}
\alpha = (\underbrace{t,t,\ldots,t}_{m_t},\underbrace{t-1,t-1,\ldots,t-1}_{m_{t-1}},\ldots,\underbrace{3,3,\ldots,3}_{m_3},\underbrace{2,2,\ldots,2}_{m_2}).
\end{equation}
By Lemma \ref{lemma for favorite theorem},
\begin{align*}
g(\alpha) = \,&\,2^{n-1} - 2^{n-m_2}+2^{n-m_2-1}-2^{n-m_2-m_3}\\&+2^{n-m_2-m_3-1}-2^{n-m_2-m_3-m_4}+\ldots+2^{m_t-1}-1.
\end{align*}
Since $m_t\ge 2$ because $\alpha_1=\alpha_2$, and $m_i\ge 1$ for $i\in\{2,3,\ldots,t-1\}$,
$$n-m_2>n-m_2-1\ge n-m_2-m_3 > n-m_2-m_3 \ge \cdots \ge m_t-1 \ge 1,$$
so all of the exponents in the powers of 2 are positive, and hence $g(\alpha)$ is odd. Therefore for an arbitrary $n$-tuple $\alpha$, $g(\alpha)$ is odd if and only if $\alpha\in A$.

Now, let $\alpha,\alpha'\in A$. We will show that $\alpha\ne \alpha'$ implies $g(\alpha) \ne g(\alpha')$. Write $\alpha$ as in \eqref{alpha1} and let $\alpha'$ be
\begin{equation}\label{alpha2}
\alpha' = (\underbrace{s,s,\ldots,s}_{m'_s},\underbrace{s-1,s-1,\ldots,s-1}_{m'_{s-1}},\ldots,\underbrace{3,3,\ldots,3}_{m'_3},\underbrace{2,2,\ldots,2}_{m'_2}).
\end{equation}
Then
$$g(\alpha') = 2^{n-1} - 2^{n-m'_2}+2^{n-m'_2-1}-2^{n-m'_2-m'_3}+\ldots+2^{m'_s-1}-1.$$
Suppose $g(\alpha) = g(\alpha')$. Let's prove $m_2= m'_2$. Without loss of generality, suppose $m_2 > m'_2$. From the proof of Lemma \ref{lemma for favorite theorem}, we can see that
$$g(\alpha) = g(\alpha_1-1, \alpha_2-1,\ldots, \alpha_{n-m_2}-1) + 2^{n-1} - 2^{n-m_2},$$
and
$$g(\alpha') = g(\alpha'_1-1, \alpha'_2-1,\ldots, \alpha'_{n-m'_2}-1) + 2^{n-1} - 2^{n-m'_2}.$$
Then
$$g(\alpha')\le 2^{n-m'_2-1} + 2^{n-1} - 2^{n-m'_2} = 2^{n-1}-2^{n-m'_2-1}-1.$$
But $g(\alpha)\ge 2^{n-1} - 2^{n-m_2}\ge 2^{n-1}-2^{n-m'_2-1} > g(\alpha')$.  Therefore $m_2 = m'_2$. But now, we have
$$g(\alpha_1-1, \alpha_2-1,\ldots, \alpha_{n-m_2}-1) = g(\alpha'_1-1,\alpha'_2-1,\ldots, \alpha'_{n-m_2}-1),$$
so the same argument shows $m_3 = m'_3$, and by induction, we conclude $\alpha = \alpha'$.

Therefore, all elements $\alpha\in A$ map (under $g$) to different odd numbers less than or equal to $2^{n-1}$. There are $2^{n-2}$ elements in $A$ and there are $2^{n-2}$ odd numbers in the set $\{1,2,\ldots, 2^{n-1}\}$. We also know that if $\alpha\not\in A$, then $g(\alpha)$ is not odd. Therefore, for every odd number $2t-1$ in the interval $[1, 2^{n-1}]$, there is exactly one $\alpha \in \{1,2,\ldots,n\}^n$ such that $g(\alpha) = 2t-1$. Therefore
$$\mathbb{P}(\alpha \text{ parks}) = \frac{2t-1}{2^{n-1}}.$$
\end{proof}

We finish the section by proving the part of pattern \ref{pattern 3} we hadn't proved.
\begin{theorem}
Given a number of the form $a/2^{n-1}$ with $0\le a\le 2^{n-1}$, there is at least one $\alpha\in\{1,2,\ldots,n\}^n$ such that $\mathbb{P}(\alpha \text{ parks}) = a/2^{n-1}$.
\end{theorem}
\begin{proof}
It is easy to verify the statement is true for small $n$, for example, for $n = 7$, Table \ref{table distribution} proves it for $a < 2^6$, for $a = 2^6$ it follows from the fact that there are $8^6>0$ parking functions. We may assume the statement is true for some $n-1$ and we want to prove it for $n$.

Theorem \ref{favorite thm} implies the statement is true for $a$ odd. Suppose $a$ is even, then $a = 2a'$ for some integer $a'\le 2^{n-2}$. Then, by the induction hypothesis, there is an $(n-1)$-tuple $\alpha'\in\{1,2,\ldots,n-1\}^{n-1}$ such that $\mathbb{P}(\alpha' \text{ parks}) = a'/2^{n-2}$. Let $$\alpha=(1,\alpha'_1+1,\alpha'_2+1,...,\alpha'_{n-1}+1).$$
The probability that $\alpha$ parks is the same as the probability that $\alpha'$ parks because after $C_1$ takes spot 1, then with $\alpha_i = \alpha'_{i-1}+1$, the cars $C_2, C_3,\ldots, C_n$ can only take spots between 2 and $n$, since their preferences are shifted by 1, it is as if the preferences were $\alpha'$ and they wanted to park on spots from 1 to $n-1$. Therefore
$$\mathbb{P}(\alpha \text{ parks}) = \frac{a'}{2^{n-2}} = \frac{2a'}{2^{n-1}} = \frac{a}{2^{n-1}}.$$
\end{proof}

\section{Comparing Naples parking to random-Naples and some asymptotics}\label{section ineqs}
Theorem \ref{thm comparing} is basically showing that $T_p(n)$ is bounded above by the line with parameter $p$ connecting $N(n)$ to $(n+1)^{n-1}$, which would give a naive estimate of what $T_p(n)$ could be.

Let $P(n) = (n+1)^{n-1}$ be the number of parking functions. The following recursive formula for $P(n)$ has been proved in \cite{Konheim} and \cite{Harris-Naples}, but can also be seen as plugging in $p = 0$ into our recursive formula \eqref{recursive eq} for $T_{p,k}(n)$ (because if $p = 0$, then for $\alpha$ to park, $\alpha$ must be a parking function)
\begin{equation}\label{recursion parking}
P(n) = \sum_{i=0}^{n-1}{{n-1\choose i}P(i){(n-i)}^{n-i-2}}(i+1).
\end{equation}
Plugging $k=1$ into \eqref{naples-orig} (or plugging $p=1$, $k=1$ into \eqref{recursive eq}) we get
\begin{equation}\label{naples parking recursion}
N(n)=\sum_{i=0}^{n-1}{{n-1\choose i}N(i){(n-i)}^{n-i-2}}(i+1+\min\{1,n-i-1\}).
\end{equation}
\begin{proof}[Proof of Theorem \ref{thm comparing}] Table \ref{evs} shows that the theorem is true for $1\le n\le 8$. Suppose the theorem is true for all $i\le n-1$. That $T(n) > P(n)$ follows directly from \eqref{recursion parking} and \eqref{recursive eq}. For the upper, from the induction hypothesis, \eqref{recursion parking} and \eqref{naples parking recursion}, we have
\begin{align*}
T(n) &=\sum_{i=0}^{n-1}{{n-1\choose i}T(i){(n-i)}^{n-i-2}}(i+1+\frac{1}{2}\min\{1,n-i-1\})\\
&\le \sum_{i=0}^{n-1}{{n-1\choose i}\left(\frac{N(i)+P(i)}{2}\right){(n-i)}^{n-i-2}}(i+1+\frac{1}{2}\min\{1,n-i-1\})\\
&= \sum_{i=0}^{n-1}{{n-1\choose i}\frac{N(i)}{2}{(n-i)}^{n-i-2}}(i+1+\frac{1}{2}\min\{1,n-i-1\})\\
&\ \ \ +\sum_{i=0}^{n-1}{{n-1\choose i}\frac{P(i)}{2}{(n-i)}^{n-i-2}}(i+1+\frac{1}{2}\min\{1,n-i-1\})\\
&=\frac{N(n)+P(n)}{2} - \sum_{i=0}^{n-1}{{n-1\choose i}\frac{N(i)-P(i)}{4}{(n-i)}^{n-i-2}}\min\{1,n-i-1\}\\
&\le \frac{N(n)+P(n)}{2}.
\end{align*}
\end{proof}

\providecommand{\bysame}{\leavevmode\hbox to3em{\hrulefill}\thinspace}
\providecommand{\MR}{\relax\ifhmode\unskip\space\fi MR }
\providecommand{\MRhref}[2]{%
  \href{http://www.ams.org/mathscinet-getitem?mr=#1}{#2}
}
\providecommand{\href}[2]{#2}

\end{document}